\documentclass{article}
\usepackage{amsmath,amssymb,amsthm,mathtools}
\usepackage{hyperref}
\usepackage{enumitem}
\usepackage{theoremref}
\usepackage{float}
\usepackage{physics}

\newtheorem{theorem}{Theorem}[section]

\newtheorem{lemma}[theorem]{Lemma}
\newtheorem{corollary}[theorem]{Corollary}

\theoremstyle{definition}           

\newtheorem{remark}[theorem]{Remark}

\title{The converse of the Cowling--Obrechkoff--Thron theorem}

\author{Devon N.~Munger \and Pietro Paparella}

\begin{document}
\maketitle

\begin{abstract}
In this work, the converse of the Cowling--Obrechkoff--Thron theorem is established. In addition to its theoretical interest, the result fills a gap in the proof of Kellogg's celebrated eigenvalue inequality for matrices whose principal minors are positive or nonnegative. \\


\end{abstract}

\section{Introduction}

An $n$-by-$n$ matrix $A$ with complex entries is called a \emph{$P$ matrix} (respectively, \emph{$P_0$ matrix}) if each of its principal minors is positive (respectively, nonnegative). This class of matrices was introduced by Fiedler and Pt\'{a}k \cite{fp1966} as a common generalization of the class of \emph{$M$ matrices} and the class of \emph{positive definite matrices}.  

In 1972, Kellogg \cite[Corollary 1]{k1972} offered the following result. 

\begin{theorem}
[Kellogg]
    \thlabel{kell}
        If $\lambda = r(\cos\theta + i \sin\theta)\in \mathbb{C}$, with $\theta \in (0,2\pi]$, then $\lambda$ is an eigenvalue of a $P$ matrix if and only if 
        \begin{equation}
            \label{eigineq}
                \vert \theta - \pi \vert > \frac{\pi}{n}.     
        \end{equation}
        If $\lambda \ne 0$, then $\lambda$ is an eigenvalue of an $n$-by-$n$ $P_0$ matrix if and only if 
        \begin{equation}
            \label{eigineq2}
                \vert \theta - \pi \vert \ge \frac{\pi}{n}.     
        \end{equation}
\end{theorem}

However, as will be explained in the sequel, Kellogg's proof is incomplete (see Section \ref{sec:background}) and requires the converse of the following result. 

\begin{theorem}
[Cowling--Obrechkoff--Thron]
    \thlabel{cto}
        Let $\mu = r(\cos\alpha + i\sin\alpha) \in \mathbb{C}$, with $\alpha \in (-\pi,\pi]$. If $q$ is a polynomial of degree $n$ with nonnegative coefficients such that $q(0) \ne 0$ and $q(\mu)=0$, then $\vert \alpha \vert > {\pi}/{n}$, unless $q$ is of the form $q(t) = a_n t^n + a_0$, in which case it has a zero satisfying $\vert \alpha \vert = {\pi}/{n}$ (if $n>1$, then $q(t) = a_n t^n + a_0$ has a conjugate pair of zeros satisfying $\vert \alpha \vert = {\pi}/{n}$).
\end{theorem}

\begin{remark}
\thref{cto} was established by Cowling and Thron in 1954 \cite[Theorem 4.1]{ct1954} and is a consequence of a more general theorem established by Obrechkoff in 1923 \cite{o1923}.    
\end{remark}

If $\lambda = r(\cos\theta + i\sin\theta)$ is a nonzero eigenvalue of a $P$ or $P_0$ matrix, with $\theta \in (0,2\pi]$, then it can be shown that $-\lambda = r(\cos(\theta-\pi) + i\sin(\theta-\pi))$ is a zero of a polynomial with positive or nonnegative coefficients, respectively (see Section \ref{sec:background}). By \thref{cto}, $\vert \theta - \pi \vert > \pi/n$ or $\vert \theta - \pi \vert \ge \pi/n$, respectively. 

Reversing the steps in the preceding argument requires the converse of \thref{cto}, which is noticeably absent in Kellogg's proof of \thref{kell}.     

\begin{theorem}
    \thlabel{convcto}
        Let $\mu = r(\cos\alpha + i\sin\alpha) \in \mathbb{C}$, with $\alpha \in (-\pi,\pi]$. If $\vert \alpha \vert \ge \pi/n$, then there is a polynomial $q$ of degree $n$ with nonnegative coefficients such that $q(0) \ne 0$ and $q(\mu) = 0$. Furthermore, if $n > 1$ and $\frac{\pi}{\alpha} \notin \mathbb{Z}$, then there is a polynomial $q$ of degree $n$ with positive coefficients such that $q(\mu) = 0$.
\end{theorem}

The purpose of this work is to establish \thref{convcto} which, to the best of our knowledge, is novel and has not appeared in the literature. 

\section{Background \& Motivation} \label{sec:background}

Suppose that $A$ is an $n$-by-$n$ matrix with complex entries and eigenvalues $\lambda_1,\ldots,\lambda_n$ (repetitions included). Recall that if $E_k(A)$ denotes the sum of the $\binom{n}{k}$ \emph{principal minors of size $k$} \cite[p.~17]{hj2013} and $p_A$ denotes the characteristic polynomial of $A$, then 
\[ p_A(t) \coloneqq \det(tI-A) = \prod_{k=1}^n (t - \lambda_k) = \left( \sum_{k=0}^{n-1} (-1)^{n-k} E_{n-k}(A) t^k \right) + t^n \]
(see, e.g., Horn and Johnson \cite[p.~53]{hj2013}). 

\begin{lemma}
    \thlabel{qauxpoly}
        Let $A$ be an $n$-by-$n$ matrix with complex entries and eigenvalues $\lambda_1,\ldots,\lambda_n$ (repetitions included). If 
        \begin{equation}
            \label{qpoly}
                q_A (t) \coloneqq (-1)^n p_A(-t),
        \end{equation}
        then 
            \begin{equation*}
                q_A (t) = \prod_{k=1}^n (t + \lambda_k) = \left( \sum_{k=0}^{n-1} E_{n-k}(A) t^k \right) + t^n.
            \end{equation*} 
        Furthermore, $p_A(\lambda) = 0$ if and only if $q_A(-\lambda)=0$.
\end{lemma} 

\begin{proof}
    Since $p_A(t) = \prod_{k=1}^n (t - \lambda_k)$, it follows that 
        \[ q_A(t) = (-1)^n \prod_{k=1}^n (-t - \lambda_k) = (-1)^n \prod_{k=1}^n (-1)^n(t + \lambda_k) = \prod_{k=1}^n (t + \lambda_k), \]
    and since 
    \[ p_A(t) = \left( \sum_{k=0}^{n-1} (-1)^{n-k} E_{n-k}(A) t^k \right) + t^n, \] 
    it follows that
        \begin{align*}
            q_A(t) 
            &= (-1)^n \left( \sum_{k=0}^{n-1} (-1)^{n-k} E_{n-k}(A) (-t)^k \right) + (-1)^n (-t)^n  \\  
            &= (-1)^n \left( \sum_{k=0}^{n-1} (-1)^{n}E_{n-k}(A) t^k \right) + t^n  \\
            &= \left( \sum_{k=0}^{n-1} E_{n-k}(A) t^k \right) + t^n.  
        \end{align*}
    Finally, because $q_A (t) = (-1)^n p_A(-t)$, it is clear that $p_A(\lambda) = 0$ if and only if $q_A(-\lambda)=0$.
\end{proof}

The following result is immediate.

\begin{theorem}
    \thlabel{thm:pmatrix}
    If $A$ is a $P$ matrix (respectively, $P_0$ matrix), then the polynomial $q_A$ has positive (respectively, nonnegative) coefficients.
\end{theorem}  

Kellogg \cite[Theorem 4]{k1972} gave a necessary and sufficient condition on a multiset of complex numbers to be the spectrum of a $P$ matrix or $P_0$ matrix. 

\begin{theorem}
    \thlabel{kelltheorem4}
        If $\Lambda = \{ \lambda_1,\ldots,\lambda_n \}$ is a multiset of complex numbers, then $\Lambda$ is the spectrum of a $P$ matrix (respectively, $P_0$ matrix) if and only if the monic polynomial
            \begin{equation}
                \label{qpoly2}
                q(t) \coloneqq \prod_{k=1}^n (t+\lambda_k) = \left(\sum_{k=0}^{n-1} c_k t^k \right) + t^n    
            \end{equation}
        has positive (respectively, nonnegative) coefficients.
\end{theorem}

Kellogg \cite[p.~174]{k1972} asserts that the converse of \thref{kell} follows from \thref{kelltheorem4} alone. However, if $\lambda = r(\cos\theta + i \sin \theta)$, with $\theta \in (0,2\pi]$, and $\theta$ satisfies \eqref{eigineq} or \eqref{eigineq2}, then, in order to apply \thref{kelltheorem4}, a multiset
    \[ \Lambda = \{ \lambda_1 \coloneqq \lambda,\ldots, \lambda_n \} \]
of complex numbers is required such that the polynomial \( q(t) \coloneqq \prod_{k=1}^n (t + \lambda_k) \) has positive or, respectively, nonnegative coefficients. If 
    \[ \mu \coloneqq -\lambda = r(\cos\alpha + i\sin\alpha), \] 
with $\alpha \coloneqq \theta - \pi$, then $q(\mu) = 0$ and $\alpha \in (-\pi,\pi]$. Furthermore, since $\theta$ satisfies \eqref{eigineq} or \eqref{eigineq2}, it follows that $\vert\alpha\vert > \frac{\pi}{n}$ or $\vert\alpha\vert \ge \frac{\pi}{n}$, respectively. Thus, application of \thref{kelltheorem4} requires a polynomial with characteristics as specified in \thref{convcto}.  

\section{Ancillary Results} \label{exp_cons}

In this section, we establish several ancillary results that will be used to prove \thref{convcto}.  

\begin{lemma}
    \thlabel{sinntheta}
        If $1 \le j < k$ and \(\alpha \in \left[\frac{\pi}{k},\frac{\pi}{k-1} \right) \), then \(\sin k \alpha \le 0\), \(\sin j \alpha > 0\), and \( \sin (k-j) \alpha > 0\). 
\end{lemma}

\begin{proof}
First, note that \(\frac{k \pi}{k-1}\leq 2\pi \Longleftrightarrow k \ge 2\). 

Since \( k \alpha \in \left[\pi, \frac{k \pi}{k-1} \right) \) and \( k \ge 2 \), it follows that $k \alpha \in [\pi, 2\pi)$. Thus, $\sin k\alpha \leq 0$. 

Similarly, $j\alpha \in \left[\frac{j\pi}{k},\frac{j\pi}{k-1} \right)$ and since $0 < j < k$, it follows that $0 < \frac{j\pi}{k}$ and $\frac{j\pi}{k-1} \le \pi$. Hence, $j\alpha \in \left(0,\pi\right)$ and $\sin j\alpha > 0$.

Lastly, $\sin(k-j) \alpha > 0$ because $1 \le k-j \le k$ and applying the preceding case.
\end{proof}

In order to motivate the next result, notice that if $\mu = r(\cos\alpha + i\sin \alpha)$, with $\alpha \in [\frac{\pi}{2},\pi)$, then $\sin\alpha > 0$ and $\cos\alpha \le 0$. If $Q_1(t) \coloneqq (t - \mu)(t - \bar{\mu})$, then   
\begin{align*}
    Q_1 (t) 
    = (t - \mu)(t - \bar{\mu})         
    = t^2 - 2 \Re\mu t + \mu\bar{\mu}   
    = t^2 - 2r \cos\alpha t + r^2,                   
\end{align*}
$Q_1(0) \ne 0$, $Q_1(\mu)=0$, and $Q_1$ has nonnegative coefficients. Furthermore, 
\[ Q_1(t) = t^2 - \frac{\sin2\alpha}{\sin\alpha} r t + r^2 \frac{\sin\alpha}{\sin\alpha}. \] 
The observations above generalize as follows.

\begin{theorem}
    \thlabel{qjpoly}
        Suppose that $1 \le j < k$ and $\mu = r(\cos \alpha + i \sin \alpha) \in \mathbb{C}$, with $\alpha \in \left[\frac{\pi}{k},\frac{\pi}{k-1} \right)$. If 
            \begin{align}
                \label{qmonic}
                Q_j(t) \coloneqq t^k - \frac{\sin k \alpha}{\sin j\alpha} r^{k-j} t^j + \frac{\sin (k-j) \alpha}{\sin j\alpha} r^k,     
            \end{align}
        then $Q_j$ has nonnegative coefficients, $Q_j (0) \ne 0$, and $Q_j(\mu) = 0$.   
\end{theorem}

\begin{proof}
    By \thref{sinntheta}, \(\sin k \alpha \le 0\), \(\sin j \alpha > 0\), and \( \sin (k-j) \alpha > 0\). Thus, $Q_j$ has nonnegative coefficients and $Q_j (0) \ne 0$.
    
    For ease of notation, let $s_n \coloneqq \sin n \alpha$ and $c_n \coloneqq \sin n \alpha$. With these conventions in mind, note that \( s_m c_n - c_m s_n = s_{m-n}\) and \( s_n + s_{-n} = 0\). Thus, 
    \begin{align*}
        \frac{s_j}{r^k} Q_k(\mu) 
        &= s_j(c_k + i s_k) - s_k(c_j + is_j) + s_{k-j} \\
        &= s_jc_k + is_js_k - c_j s_k - is_js_k + s_{k-j}  \\
        &= s_jc_k - c_j s_k + s_{k-j}                  \\
        &= s_{j - k} + s_{k-j} = 0,
    \end{align*}
i.e., $Q_j(\mu) = 0$.
\end{proof}

\begin{corollary}
    \thlabel{pospoly}
        Suppose that $k \ge 2$ and $\mu = r(\cos \alpha + i \sin \alpha) \in \mathbb{C}$, with $\alpha \in \left(\frac{\pi}{k},\frac{\pi}{k-1} \right)$. If 
        \begin{equation}
            \label{bigQ}
                Q(t) := \frac{1}{k-1}\sum_{j=1}^{k-1} Q_j(t)
        \end{equation}
        then $Q$ is a monic polynomial of degree $k$ with positive coefficients and $Q(\mu) = 0$.
\end{corollary}

\begin{proof}
    By \thref{qjpoly},  
    \[ Q(\mu) = \frac{1}{k-1}\sum_{j=1}^{k-1} Q_j(\mu) = 0. \]
    Since 
    \begin{align*}
            Q(t) 
            = \frac{1}{k-1} \left( \sum_{j=1}^{k-1} \frac{\sin (k-j) \alpha}{\sin j\alpha} r^k - \sum_{j=1}^{k-1} \frac{\sin k \alpha}{\sin j\alpha} r^{k-j} t^j \right) + t^k. 
    \end{align*}
    it follows that $Q$ is monic and of degree $k$. 
    
    If $1 \le j \le k-1$, then, by \thref{sinntheta}, \(\sin k \alpha \le 0\), \(\sin j \alpha > 0\), and \( \sin (k-j) \alpha > 0\). Moreover, since \( k \alpha \in \left(\pi, \frac{k \pi}{k-1} \right) \) and \( k \ge 2 \), it follows that $k \alpha \in (\pi, 2\pi)$. Thus, $\sin k\alpha < 0$. Consequently, 
        \[ -\frac{\sin k \alpha}{\sin j\alpha} r^{k-j} > 0 \]
    and
        \[ \frac{\sin (k-j) \alpha}{\sin j\alpha} r^k > 0, \]
    i.e., the polynomial $Q$ has positive coefficients. 
\end{proof}

\section{Proof of the Main Result}

We are now ready to prove \thref{convcto}.

\begin{proof}
    [Proof of \thref{convcto}]
If $n=1$ and $\vert \alpha \vert \ge \pi$, then $\alpha = \pi$, $\mu < 0$, and the desired polynomial is $t - \mu$. 

Otherwise, assume that $n > 1$ and $\vert \alpha \vert \ge \pi/n$. It suffices to show that $\mu$ is a zero of a polynomial possessing the desired properties of degree $k$, with $1 \le k \le n$ (if $1 \le k < n$, then we can multiply the polynomial by $t^{n-k} + 1$ to obtain a polynomial of degree $n$ with the desired properties). If $\alpha = \pi$, then $t - \mu$ is the desired polynomial. Otherwise, it may be assumed, without loss of generality, that $\pi/n \le \alpha < \pi$, since zeros of polynomials with real coefficients appear in conjugate pairs. If $k \coloneqq \lceil \frac{\pi}{\alpha} \rceil$, then $2 \le k \le n$ and $\alpha \in \left[ \frac{\pi}{k}, \frac{\pi}{k-1} \right)$. By \thref{qjpoly}, the polynomial $Q_j$ defined in \eqref{qmonic} has the desired properties.

Further suppose that $n > 1$ and $\frac{\pi}{\alpha} \notin \mathbb{Z}$. It suffices to show that $\mu$ is a zero of a polynomial possessing the desired properties of degree $k$, with $1 \le k \le n$ (if $1 \le k < n$, then we can multiply the polynomial by $\sum_{j = 0}^{n-k} t^j$ to obtain a polynomial of degree $n$ with the desired properties). It may be assumed, without loss of generality, that $\pi/n < \alpha < \pi$ and $\frac{\pi}{\alpha} \notin \mathbb{N}$, since zeros of
polynomials with real coefficients appear in conjugate pairs. If $k \coloneqq \lceil \frac{\pi}{\alpha} \rceil$, then $2 \le k \le n$ and $\alpha \in \left( \frac{\pi}{k}, \frac{\pi}{k-1} \right)$. By \thref{pospoly}, the polynomial $Q$ defined in \eqref{bigQ} has the desired properties.
\end{proof}



\bibliographystyle{abbrv}
\bibliography{refs}

\end{document}